\documentclass[11pt]{amsart}

\usepackage{amsmath}
\usepackage{amssymb}

\usepackage{amsfonts}
\usepackage{amsthm}
\usepackage{enumerate}
\usepackage[mathscr]{eucal}

\usepackage{tikz}

\usepackage{bbm}
\usepackage{graphicx}
\usepackage{verbatim}
\def\compl{\complement}

\newcommand{\per}{{\text{\it per}}}

\newcommand{\Be}{\begin{equation}}
\newcommand{\Ee}{\end{equation}}

\newcommand{\Bea}{\begin{eqnarray}}
\newcommand{\Eea}{\end{eqnarray}}
\newcommand{\Beas}{\begin{eqnarray*}}
\newcommand{\Eeas}{\end{eqnarray*}}
\newcommand{\Benu}{\begin{enumerate}}
\newcommand{\Eenu}{\end{enumerate}}
\newcommand{\Bi}{\begin{itemize}}
\newcommand{\Ei}{\end{itemize}}

\def\intslash{\rlap{\kern  .32em $\mspace {.5mu}\backslash$ }\int}
\def\fint{\intslash}

\def\qsl{{\rlap{\kern  .32em $\mspace {.5mu}\backslash$ }\int_{Q_x}}}

\def\cf{{\it cf}}

\def\loc{{\text{\rm loc}}}

\def\dist{{\text{\it dist}}}

\def\inn#1#2{\langle#1,#2\rangle}
\def\biginn#1#2{\big\langle#1,#2\big\rangle}

\def\noi{\noindent}

\def\lc{\lesssim}
\def\gc{\gtrsim}

\def\eps{\varepsilon}

\def\ka{\kappa}

\def\la{\lambda}

              \def\Om{\Omega}

\def\fM{{\mathfrak {M}}}

\def\fS{{\mathfrak {S}}}

\def\bbR{{\mathbb {R}}}

\def\bbT{{\mathbb {T}}}

\def\bbZ{{\mathbb {Z}}}

\def\cA{{\mathcal {A}}}
\def\cB{{\mathcal {B}}}
\def\cC{{\mathcal {C}}}

\def\cE{{\mathcal {E}}}

\def\cI{{\mathcal {I}}}

\def\cK{{\mathcal {K}}}

\def\cM{{\mathcal {M}}}
\def\cN{{\mathcal {N}}}

\def\cU{{\mathcal {U}}}
\def\cV{{\mathcal {V}}}

\def\cY{{\mathcal {Y}}}
\def\cZ{{\mathcal {Z}}}

\def\R{{\hbox{\bf R}}}

\def\I{{\hbox{\bf I}}}
\def\J{{\hbox{\bf J}}}

 at 10 true pt

\def\be#1{\begin{equation}\label{ #1}}
\def\endeq{\end{equation}}
\def\endal{\end{align}}
\def\bas{\begin{align*}}
\def\eas{\end{align*}}
\def\bi{\begin{itemize}}
\def\ei{\end{itemize}}

\def\eps{\varepsilon}

\def\emph#1{{\it #1}}
\def\textbf#1{{\bf #1}}

\usepackage{bbm}
\def\bbone{{\mathbbm 1}}

\theoremstyle{plain}
  \newtheorem{theorem}{Theorem}[section]

   \newtheorem{prop}[theorem]{Proposition}
   
   \newtheorem{lemma}[theorem]{Lemma}
   \newtheorem{corollary}[theorem]{Corollary}

\theoremstyle{remark}
   \newtheorem{remark}[theorem]{Remark}
   \newtheorem{remarks}[theorem]{Remarks}

\theoremstyle{definition}

\begin{document}

\title[Singular integrals and mixing flows]{Singular integrals and  \\ a problem on mixing flows}

\author[M. Had\v zi\'c   \ \ \  A. Seeger  \ \ \  C. Smart \ \ \   B. Street]
{Mahir Had\v zi\'c   \ \    Andreas Seeger  \ \  Charles K. Smart \ \    Brian Street}

\address{Mahir Had\v zi\'c\\Department of Mathematics\\
King's College London\\ 
Strand, London WC2R 2LS, UK}
\email{mahir.hadzic@kcl.ac.uk}

\address{Andreas Seeger \\ Department of Mathematics \\ University of Wisconsin \\480 Lincoln Drive\\ Madison, WI, 53706, USA} \email{seeger@math.wisc.edu}

\address{Charles K. Smart\\ Department of Mathematics \\ University of Chicago
\\734 S. University Avenue\\
Chicago, IL 60637, USA}
\email{smart@math.uchicago.edu}

\address{Brian Street \\ Department of Mathematics \\ University of Wisconsin \\480 Lincoln Drive\\ Madison, WI, 53706, USA} \email{street@math.wisc.edu}

\begin{abstract}  We prove a result related to Bressan's mixing problem. We establish an inequality for the change of Bianchini semi-norms of characteristic functions under the  flow generated by a divergence free time dependent vector field. The approach leads to a bilinear singular integral operator  for which we prove bounds  on Hardy spaces. We include additional observations about the approach and  a discrete toy version of Bressan's problem.

\end{abstract}


\thanks{Research supported in part by  National Science Foundation grants}
\maketitle
\section{Introduction}\label{intro}


\subsection{\it Mixing flows}\label{mixingflows} 


We consider subsets $A$ of $\bbT^d\equiv \bbR^d/\bbZ^d$. For $0<r<1/4$, 
$x\in \bbR^d$  let  $B_r(x)$  denote the ball of radius 
$r$ centered at $x$, with respect to the usual  geodesic distance on  $\bbT^d$.
A measurable set $E\subset \bbT^d$ is {\it mixed  at scale $r$,} with mixing constant 
$\ka\in (0,1/2)$, if 
\begin{equation}\label{mixing}\ka
 \,\le\,\frac{ |E\cap B_r(x)| }{ |B_r(x)|}\,\le\,  1-\ka ,\quad 
 \forall
 x\in \bbT^d.\end{equation}


Let $v$  
be a time-dependent,  a priori  
 smooth vector field, defined on $\bbT^d\times [0,T]$ with values in the  tangent bundle of the torus. The vector field can be considered a vector field 
$(x,t)\mapsto v(x,t)$ on $\bbR^d$ 
which is periodic in $x$, i.e. 
$$v(x+k,t)=v(x,t) \text{ for all } (x,t)\in \bbR^d\times \bbR, \
 k\in \bbZ^d\,.$$ We assume that 
$$\text{div}_x v(x,t)=0$$
and let  $\Phi $ be the  flow generated by $v$. I.e. $\Phi$ satisfies 
$$
\begin{gathered}\frac{\partial}{\partial t} \Phi(x,t)= v(\Phi(x,t),t), \\
\Phi(x,0)=x.
\,\end{gathered}$$
For every $t$ the map $x\mapsto \Phi(x,t)$ is a volume  preserving diffeomorphism on
$\bbR^d$
satisfying  $$\Phi(x+k,t)-k=\Phi(x,t),\quad \text{$x\in \bbR^d$, $k\in \bbZ^d$.}$$
In what follows we shall also use the notation
$\Phi_t(x)=\Phi(x,t)$. 
We are interested in mixing flows which transport an unmixed set $\Om$ at time $t=0$ to a set $\Phi_T(\Om)$
mixed at scale $\eps$ at time $t=T$.
\subsection{\it Bressan's problem}\label{bressanprbl}
Split  $\bbT^d$ as $\Omega_L\cup \Omega_R$ with $\Omega_R=\Omega_L^\complement$  where
\Be\label{halftorus}\Omega_L=\{x: 0\le x_1<\frac12\},  \quad
\Om_R=\{x: \frac 12\le x_1<1\}. \Ee
Let $0<\eps<1/4$.  Consider a periodic flow $\Phi_t$ generated by a smooth time dependent divergence free vector field, and {\it assume that at time $t=T$ the flow mixes  $\Om_L$ at scale $\eps$;} i.e. the set $E=\Phi_T(\Om_L)$ satisfies  \eqref{mixing} with $r=\eps$.
Bressan \cite{bressan-website} asks  (setting $\kappa=1/3$) whether  there is a universal constant $c_d>0$ such that
\Be\label{Bressanineq} 
\int_0^T\int_{[0,1)^d}| D_x v(x,t)| \, dx\, dt \ge c_d \log(1/\eps)\,.
\Ee
As noted in \cite{bressan-website}
it suffices to consider the case $T=1$, by replacing $v(x,t)$ with $Tv(x,t/T)$.  In \cite{bressan},  Bressan formulated 
a more general conjecture for mildly compressible flows.

Bressan's conjecture is still open at the time of this writing. Therefore it is of interest to ask  for 
corresponding lower bounds if the $L^1(\bbT^d)$ norm is replaced by a larger norm. That is,  under the assumption that the flow generated by $v$ mixes the set  at scale $\eps$ 
with mixing constant $\gamma$, do we have a universal lower bound of the form
\Be\label{BressanineqY} 
\int_0^T\| D_x v(\cdot,t)\|_\cY  \,dt \ge c_\cY(\ka) \log(1/\eps)\,
\Ee
for suitable function spaces $\cY\subset L^1(\bbT^d)$
or even $\cY\subset M(\bbT^d)$ with $M(\bbT^d)$ the space of bounded Borel measures on $\bbT^d$?
Crippa and De Lellis \cite{crippa-deLellis} showed this
 for $\cY=L^p(\bbT^d)$, $1<p<\infty$  and also 
for the 
space $\cY$ consisting of functions for which the Hardy-Littlewood maximal function $M_{HL}f$ belongs to $L^1(\bbT^d)$, i.e. for $\cY=L\log L(\bbT^d)$. We shall  discuss two ways to improve $\cY$ to a local Hardy space. In \S\ref{toymodel} we consider a discrete toy problem on $\bbT^2$ for which we prove an analogue of the $L^1$ conjecture, although this toy model does not yield significant information for the general Bressan problem.
It should be noted that the lower bound $\log(1/\eps)$ is sharp and cannot even be improved by working with $L^p$ spaces, see the recent results  by Yao and Zlato\v s \cite{yao-zlatos} and by Alberti, Crippa and Mazzucato
\cite{alberti-etal}. 

\subsection{\it An approach to Bressan's problem via a Bianchini  semi-norm}\label{bianchini-section}
We denote by  $$\intslash_{B_r(x)} f(y) dy=\frac{1}{|B_r(x)|}\int_{B_r(x)} f(y) dy$$ the average of $f$ over the ball $B_r(x)$. 
For $\eps<1/8$ define the truncated Bianchini semi-norm by
$$
\|f\|_{\cB(\eps)}:= \int_{\eps}^{1/4} \int_{\bbT^d} \big|f(x)- \intslash_{B_r(x)} f(y) dy
 \big| \,dx\, \frac{dr}{r}
$$
and let the {\it Bianchini space} consist of all $L^1(\bbT^d)$ functions  for which 
$$\|f\|_{\cB}: =\sup_{\eps<1/4}\|f\|_{\cB(\eps)}
\,<\,\infty.$$ 
This space was proposed by Bianchini in \cite{bianchini} as a measure for mixing in a one-dimensional shuffling problem. There it was denoted  $\dot B^{0,1,1}$ in reference to Besov although this space 
does not actually belong to the usual scale of Besov spaces. The connection with mixing is given by the following

\noi{\it Observation:} If $E$ is mixed at scale $\eps>0$, with mixing constant $\ka$, then
$$\Big| \bbone_E(x)-\intslash_{B_r(x)}\bbone_E(y) \,dy \Big| \ge \kappa \text{ a.e.}\,\quad \forall r>\eps.$$
Hence integrating in $r$ and $x$ one  gets $$\|\bbone_E\|_{\cB}\ge \|\bbone_E\|_{\cB(\eps/2)}\gc \ka\log(1/\eps)\,.$$
Also by straightforward computation $\|\bbone_{\Om_L}\|_\cB \lc 1$ for $\Om_L$ as in \eqref{halftorus}.
Our main result is  an inequality for the change of the Bianchini norm of a characteristic function under the flow, which does not itself refer to mixing. In this result $h^1(\bbT^d)$ 
denotes the local Hardy space (\cite{goldberg}); note that for $p>1$, we have the embeddings
$L^p(\bbT^d)\subsetneq L\log L(\bbT^d) \subsetneq	 h^1(\bbT^d)\subsetneq L^1(\bbT^d).$
\begin{theorem}\label{bressan-bianch}
Let $v$, $\phi$ be as above. Then the inequality 
$$\|\bbone_{\phi_T(A)}\|_{\cB} \le
\|\bbone_{A}\|_{\cB} + C_d \int_0^T \|Dv(\cdot,t)\|_{h^1(\bbT^d)}
\,dt$$
holds 
for any measurable subset $A\subset \bbT^d$, with $C_d$ a universal constant.
\end{theorem}
Theorem \ref{bressan-bianch}  gives an alternative approach to the results by Crippa and De Lellis. 
By the above discussion  the following implication  on the mixing problem is immediate.
\begin{corollary}\label{bressanhardy}
Let $0<\eps<1/4$ and  let the vector field $v$  satisfy  the assumptions in the Bressan problem stated in \S\ref{bressanprbl}. Then inequality 
\eqref{BressanineqY} holds with $\cY=h^1(\bbT^d)$.
\end{corollary} 

A weaker form of Theorem \ref{bressan-bianch}, with $D v(\cdot,t)\in L^p$, $p>1$,  was cited  in 
\cite[eq.(1.5)]{sss} with reference to the current project, and served as initial motivation for the harmonic analysis results of that paper. 
 Flavien  L\'eger \cite{leger} independently found a related approach to mixing which leads to a limiting version of  the singular 
integral forms in \eqref{fSepsRdef} below. Instead of the change of the Bianchini norm of characteristic functions he considers the change of the  square of a logarithmic
$L^2$-Sobolev norm of an arbitrary passive scalar advected under a divergence free vector field. For more comments about this see  \S\ref{legersection} below.

\subsection*{\it This paper}
A computation reducing the problem to an inequality for bilinear singular integral operators is given in \S\ref{SIreductionmain}. In \S\ref{cj-reduction} we recall  
 the connection with 
Christ-Journ\'e operators. In  \S\ref{mainresults} we describe  the natural decomposition of our singular integral form and state the two main propositions \ref{mainfirst} and \ref{easy} which lead to  $h^1\to L^1$ boundedness. These propositions 
are proved in \S\ref{proofofmainfirst} and \S\ref{proofofeasy}. In
\S\ref{additional} we make additional remarks about the approach by Crippa and De Lellis and the results by L\'eger. 
In \S\ref{nonL1} we prove a result concerning the (non)-feasability of  the singular integral estimate for Bressan's $L^1$ conjecture and formulate a related discrete problem. 
Finally, in \S\ref{toymodel} we include  some positive results on  a toy model for the $L^1$ version of Bressan's conjecture.

\

\noindent{\it Acknowledgement.} We thank the referees for their suggestions.

\section{The reduction to singular integrals} \label{SIreduction}
\subsection{\it The main computation}\label{SIreductionmain}
Given $A\subset \bbT^d$ we define
$$f_A(x)=
\bbone
_{A}(x)-\bbone_{A^\compl}(x).$$
Since constants (and thus
$\bbone_{A}+\bbone_{A^\compl}$) have semi-norm equal to $0$ in $\cB(\eps)$ we have
$\|\bbone_A\|_\cB=\|\bbone_{A^\complement}\|_\cB$ and thus
\Be\label{1AvsfA}
\|\bbone_A\|_{\cB(\eps)}=\frac 12 \|f_A\|_{\cB(\eps)}\,.
\Ee

For a  periodic time independent vector field $ b$
and  functions $f$,  $g$ on $\bbT^d$ we define
\Be\label{fSepsdef}
\fS_{\eps}^\per[f,g,b]=\iint_{\substack{ (x,y)\in \bbT^d\times\bbT^d\\\eps\le |x-y|\le 1/4} }\frac{\inn{x-y}{b(x)-b(y)}}{|x-y|^{d+2}} g(y) f(x)\,dy\,dx.
\Ee

\begin{prop}\label{bianchprop}
Let $v$, $\phi$ be as in the introduction.
Then
$$\big\|\bbone_{\Phi_T(A)}\big\|_{\cB_\eps}-
\big\|\bbone_{A}\big\|_{\cB_\eps} = \frac{1}{2V_d}\int_0^T 
\fS_\eps^\per\big[f_{\Phi_t(A)}, f_{\Phi_t(A)}, v(\cdot, t)\big]\,dt
$$
where $V_d$ denotes the volume of the unit ball in $\bbR^d$.
\end{prop}

\begin{proof}
We compute using the incompressibility of the flow,
$$
\begin{aligned}
&\|f_A\circ\Phi_T^{-1}\|_{\cB(\eps)}-
\|f_A\|_{\cB(\eps)}
\\&=\,
\int_{\eps}^{1/4}\Big[ \int_{\bbT^d}
\big|f_A(\Phi_T^{-1}(x))- \intslash_{B_r(x)}
f_A(\Phi_T^{-1}(y)) dy \big| \,dx\,
\\&\quad\qquad\qquad\qquad\,-\,
\int \big|f_A(x)- \intslash_{B_r(x)} f_A(y) dy \big| \,dx\Big]\, \frac{dr}{r}
\\
&=\,
\int_{\eps}^{1/4}\Big[ \int_{\bbT^d}
\big|f_A(z)- \intslash_{\Phi_T^{-1}B_r(\Phi_T(z))}
f_A(w) dw \big| \,dz\,
\\&\quad\qquad\qquad\qquad\,-\,
\int \big|f_A(x)- \intslash_{B_r(x)} f_A(y) dy \big| \,dx\Big]\, \frac{dr}{r}\,.
\end{aligned}
$$
Now $f_A(x)=1$ for $x\in A$ and $f_A(x)=-1$ for $x\in A^\compl$. 
Thus from the above, as we use that $- |E|\le \int_E f_A(y)dy\le |E|$ for all measurable sets $E$,   we obtain
$$
\begin{aligned}
&\|f_A\circ\Phi_T^{-1}\|_{\cB(\eps)}-
\|f_A\|_{\cB(\eps)}\,=\,
\\
&\qquad
\int_{\eps}^{1/4}\Big\{ \int_{A}
\Big[f_A(x)- \intslash_{\Phi_T^{-1}B_r(\Phi_T(x))}
f_A(y) dy\Big]  \,dx
\\ & \qquad \,-\,
\int_A\Big  [f_A(x)- \intslash_{B_r(x)} f_A(y) dy\Big ]
\,dx
\\& \qquad \,+\,
\int_{A^\compl} \Big [
\intslash_{\Phi_T^{-1}B_r(\Phi_T(x))}
f_A(y) dy -f_A(x)\Big]\,dx
\\ & \qquad \,-\,
\int_{A^\compl} \Big [\intslash_{B_r(x)} f_A(y) dy
- f_A(x) \Big ]
\,dx \,\Big\}\,
 \frac{dr}{r}
\end{aligned}
$$
and this implies
\begin{align}
&\|f_A\circ\Phi_T^{-1}\|_{\cB(\eps)}-
\|f_A\|_{\cB(\eps)} \notag
\\
&=\,
\int_{\eps}^{1/4}
\int f_A(x) \Big[
 \intslash_{B_r(x)} f_A(y) dy
-\intslash_{\Phi_T^{-1}B_r(\Phi_T(x))}
f_A(y) dy
 \Big]\,dx\,\frac{dr}{r} \notag
\\&=\,-
\int f_A(x) \,\int_0^T \frac{d}{dt}\Big[
\int_{\eps}^{1/4}\intslash_{\Phi_t^{-1}B_r(\Phi_t(x))}
f_A(y) dy \,\frac{dr}{r}\Big]
\, dt\, dx
\label{derivative}
\end{align}
Now let $V_d$ denote the measure  of the unit ball in $\bbR^d$. Then
$$
\begin{aligned}
&\int_{\eps}^{1/4}\intslash_{\Phi_t^{-1}B_r(\Phi_t(x))}
f_A(y) dy \,\frac{dr}{r}
\\
&=V_d^{-1}\int_{\eps}^{1/4}r^{-d-1}\int_{\{y: |\Phi_t(x)-\Phi_t(y)|\le r\}}
f_A(y) dy \,dr
\\&=
\int H_\eps(\Phi_t(x)-\Phi_t(y)) f_A(y) \,dy
\end{aligned}
$$
where
$$
H_\eps(u)=\begin{cases}
d^{-1}V_d^{-1} (\eps^{-d}-(1/4)^{-d}) &\text{ if } |u|\le \eps
\\
d^{-1}V_d^{-1} (|u|^{-d}-(1/4)^{-d}) &\text{ if } \eps<|u|\le 1/4
\\
0 &\text{ if } |u|> 1/4
\end{cases}
$$
 $H_\eps$ is a Lipschitz function, and  has a bounded gradient given by
$$
\nabla H_\eps(u)= -V_d^{-1}\frac{u}{|u|^{d+2}}\bbone_{\cA(\eps,1/4)}(u)
$$
where $\cA(\eps,1/4)(u)=\{u\in \bbR^d: \eps\le |u|\le 1/4\}$. Thus
$$
\begin{aligned}
&\frac{d}{dt}\Big[
\int_{\eps}^{1/4}\intslash_{\Phi_t^{-1}B_r(\Phi_t(x))}
f_A(y) dy \,\frac{dr}{r}\Big]
\\=&\,
\int \biginn{\tfrac{d}{dt}(\Phi_t(x)-\Phi_t(y))
}{\nabla H_\eps (\Phi_t(x)-\Phi_t(y))} f_A(y) \, dy
\\=&\,-
\int\limits_{\eps\le |\Phi_t(x)-\Phi_t(y)|
\le \frac 14}
f_A(y)
\frac{
\inn{v(\Phi_t(x),t)- v(\Phi_t(y),t)}{ \Phi_t(x)-\Phi_t(y) }}
{V_d|\Phi_t(x)-\Phi_t(y)|^{d+2}}\,dy.
\end{aligned}
$$
Using this in \eqref{derivative} and changing variables we obtain
$$
\begin{aligned}
&\|f_{\Phi_T (A)}\|_{\cB(\eps)}-
\|f_A\|_{\cB(\eps)} \notag
\\
&=\int_0^T\iint\limits_{\substack{ \eps\le |x-y|\le \frac 14}}
f_{\Phi_t(A)}(x)f_{\Phi_t(A)}(y)
\frac{
\inn{v(x,t)- v(y,t)}{ x-y} }
{V_d|x-y|^{d+2}}\,dy\,dx\, dt
\end{aligned} 
$$ which gives the assertion.
\end{proof}

In order to complete the proof of  Theorem \ref{bressan-bianch} it suffices to prove, for divergence free vector fields $b$, the inequality
\Be \label{ineqcharfctTd}
\big|  \fS_\eps^\per[\bbone_A, \bbone_B,b] \big| \lc \| D b\|_{h^1(\bbT)} 
\Ee
for measurable subsets $A, B\subset \bbT^d$
and apply Proposition \ref{bianchprop}.
Without loss of generality (after localization) one can assume that the diameters of $A$ and $B$ are  small. 
We can then 
 transfer the problem to $\bbR^d$ and look at the analogous singular integral form
  on $\bbR^d$,
defined by 
  \Be\label{fSepsRdef}
\fS_{\eps,R}[f, g, b]\,=\,
\iint\limits_{\eps\le |x-y|\le R} \frac{\inn{x-y}{b(x)-b(y)}}{|x-y|^{d+2}} g(y) f(x)\, dy\,
dx.
\Ee
Now   \eqref{ineqcharfctTd} 
follows from  
\begin{theorem}\label{mainsingularintegral} 
(i) For $\eps<R$,
\Be\label{trilinearformestepsR}
\big| \fS_{\eps,R}[f,g,b]\big|\le C_d \|Db\|_{H^1(\bbR^d)} \|g\|_\infty\|f\|_\infty
\Ee
with $C_d$ independent of $\eps, R$.

(ii) If in addition $R<1$ the Hardy space $H^1$ may be replaced in \eqref{trilinearformestepsR}
with the local Hardy space $h^1$.
\end{theorem} 



\subsubsection*{\it Remark}
An examination of the proof of Theorem \ref{mainsingularintegral} also 
shows that for  $f,g\in L^\infty$, $Db\in H^1$,
\Be \label{pv}\lim_{\substack {\eps\to 0\\R\to\infty}}\fS_{\eps,R}[f, g,b] = \fS[f, g,b]\Ee
where $\fS$ is  a singular integral form satisfying
\Be\label{trilinearformest}
\big| \fS[f,g,b]\big|\le C_d \|Db\|_{H^1} \|g\|_\infty\|f\|_\infty.
\Ee

\subsection{\it Connection with  Christ-Journ\'e operators }\label{cj-reduction}
There is a close relation  with the operators considered by Christ and Journ\'e \cite{cj}, and in more generality by three of the authors \cite{sss}. The result of Proposition \ref{bianchprop} is cited in \cite{sss} and served as a motivation for the harmonic analysis results of that paper.

For $\beta\in L^1_\loc$ we can define for almost every pair $(x,y)\in \bbR^d\times \bbR^d$   
\Be\label{mxy} m_{x,y}[\beta]= \int_0^1 \beta(sx+(1-s)y) \,ds,\,\Ee
the mean of $\beta$ over the line segment connecting the points $x$ and $y$.
Given a Calder\'on-Zygmund convolution kernel $K$ in $\bbR^d$, $d\ge 2$,  
and $a\in L^\infty(\bbR^d)$ (or $L^q(\bbR^d)$)  the so called
$d$-commutator   of first order 
 is defined by 
\Be\cC_{K}[f,\beta](x)=  \int_{\bbR^d} K(x-y)  m_{x,y}[\beta] \, f(y)\,dy.\Ee
For a divergence free vector field $b$  set 
\Be\label{alphaij}\beta_{ij}= \frac{\partial b_i}{\partial x_j}\Ee
so that
$$b_i(x)-b_i(y)= \sum_{j=1}^d(x_j-y_j)m_{x,y}[\beta_{ij}].
$$
By the assumption
$\text{div} (b)=0$ we have   $\beta_{dd}(x)=-\sum_{i=1}^{d-1} \beta_{ii}(x)$; hence
\Be \label{KiKij}
\frac{\inn{x-y}{b(x)-b(y)}}{|x-y|^{d+2}} =
\sum_{i=1}^{d-1} K_i(x-y) m_{x,y}[\beta_{ii}]+
\sum_{\substack{1\le i,j\le d\\i\neq j}}
K_{ij}(x)m_{x,y} [\beta_{ij}]
\Ee
where 
\begin{subequations}
\Be\label{Ki}
K_i(x)= \frac{(x_i-y_i)^2-(x_d-y_d)^2}{|x-y|^{d+2}}\Ee
and
\Be\label{Kij}K_{ij}(x)= \frac{(x_i-y_i)(x_j-y_j)}{|x-y|^{d+2}}\,.\Ee
\end{subequations} Consequently,
\Be \label{CJdecomp}
\fS[f,g,b]= \sum_{i=1}^{d-1} \int \cC_{K_i}[g,\beta_{ii}](x) f(x) \,dx+
\sum_{i\neq j} \int \cC_{K_{ij}}[g, \beta_{ij}](x) f(x)\, dx\,.
\Ee
This  identity turns our problem  into a problem on $d$-commutators. Note that $K_i(x)$ and $K_{ij}(x)$ above are of the form
$ \Omega(x/|x|)|x|^{-d}$ 
where $\Omega\in C^\infty(S^{d-1})$ is even  with
$\int_{S^{d-1}}\Omega(\theta) d\theta=0.$
From \eqref{CJdecomp} and the results in \cite{sss} one obtains
\Be \label{LpforS}
|\fS[f,g,b]|\le C(p_1, p_2,p_3) \|f\|_{p_1} \|g\|_{p_2}\|Db\|_{p_3}
\Ee
for $p_1^{-1}+p_2^{-1}+p_3^{-1}=1$, with $1<p_i\le \infty$. S.   Hofmann suggested in personal communication  that this result 
might  also follow from (the isotropic version) of  his off-diagonal $T(1)$ theorem  in
 \cite{hofmann-off-diagonal}.
These  results do not seem to give enough information in the  case $p_3=1$ which is relevant for the focus of this paper.
The weak type $(1,1)$ result in \cite{seeger-rev} can be modified to see that for $g\in L^\infty$ and $\beta\in L^1$ we have
$\cC_K[g,\beta]\in L^{1,\infty}$ and this can be used to prove a bound for compactly supported $b$ with $Db\in L\log L$; however there does not seem to be an  $H^1\to L^1$ result for $d$-commutators 
which can be used to establish 
Theorem \ref{mainsingularintegral}.
 Our approach will  be more direct; we  rely on
some regularizations for the kernels, and use 
 the original $T(1)$ theorem by  David and Journ\'e for one of the terms and Littlewood-Paley estimates for the others. The  atomic decomposition will be used for the Hardy space estimates.

  \subsection{\it Further reductions}\label{mainresults}
  We now begin with the proof of Theorem \ref{mainsingularintegral} and first make an easy 
  observation about single scale contributions.  Using 
$$ b(x)- b(y)= \int_0^1 D b (sx+(1-s)y) ds \, (x-y)$$ we observe,
using a straightforward application of H\"older's inequality, 
 that for  each $R>0$
\Be\label{singleannulus}
\iint_{R\le |x-y|\le 2R} \frac{\big|\inn{x-y}{ b(x)- b(y)}|}{|x-y|^{d+2}}  |g(y)| |h(x) |dy\, dx 
\lc  \|Db\|_{p_1}\|g\|_{p_2} \|h\|_{p_3}, 
\Ee
for $p_1^{-1}+p_2^{-1} + p_3^{-1}=1$,  $1\le p_1,p_2, p_3\le \infty$.

Let $\chi$ be a {\it radial} $C^\infty$ function supported in $\{x:1/2< |x|<2\}$ such that
$\sum_{k\in \bbZ} \chi(2^k x)=1$ for $x\neq 0$.
Define 
\[\chi_k(x)= \chi(2^kx)\] and 
set 
\Be\label{Skdef}S_k[g,  b](x)\,=\,
\int \chi_k(x-y)\frac{\inn{x-y}{ b(x)- b(y)}}{|x-y|^{d+2}} g(y) dy\,.
\Ee
Using  \eqref{singleannulus} 
it  is  easy to see that Theorem \ref{mainsingularintegral}  follows from
\begin{theorem}\label{sumhSk}
\Be \label{ksum}
\Big|\sum_{k\in \cZ} \int h(x) S_k [g,  b](x) dx\Big | \le C
\|Db\|_{H^1}\|g\|_\infty \|h\|_\infty
\Ee
where the summation over $k$ is over a {\it finite} set $\cZ$ of integers and the constant $C$ does not depend on the cardinality of this set.  
\end{theorem}

We need further decompositions. 
Let $\phi$ be a $C^\infty$ function with support in $\{x:|x|\le 1/2\}$ such that 
\begin{subequations}
\Be \label{intphi}\int \phi(x) dx=1\Ee
and \Be\label{intphimoment} \int \phi(x) x_i dx=0, \quad i=1,\dots, d.
\Ee
\end{subequations}
Define
\begin{align} 
\phi_k(x)&=2^{kd} \phi(2^kx)
\notag
\\
\psi_l(x)&= \phi_{l}(x)-\phi_{l-1}(x)
\notag
\end{align}
For every $k$ we have,  in the sense of distributions,
\Be\label{delta}\phi_k+\sum_{n=1}^\infty \psi_{k+n}=\delta;
\Ee  here $\delta$ is the  Dirac measure. Note that $\int \psi_l(x) \pi(x)dx=0$ for all affine linear functions $\pi$.

Theorem 
\ref{sumhSk}
follows immediately from the second parts of the following two propositions. All constants will be independent of the cardinality of $\cZ$.

\begin{prop} \label{mainfirst} (i) For
 $1<p<\infty $,
\[\Big\|\sum_{k\in \cZ} S_k [g, \phi_k* b]\Big  \|_{p} \le C(p_1,p_2)  \|g\|_{\infty} \|Db\|_{p}\,.
\]
(ii) 
\[\Big\|\sum_{k\in\cZ} S_k [g, \phi_k* b]\Big  \|_{1} \le C  \|g\|_{\infty} \|Db\|_{H^{1}}\,.
\]
\end{prop}

\begin{prop} \label{easy}
(i) Let  $1<p_1,p_2,q<\infty$ and $1/p_1+1/p_2=1/q$. Then for $n=1,2,3,\dots$
\[
\Big\|\sum_{k\in\cZ} S_k [g, \psi_{k+n}*  b]
 \Big \|_{q} \le C(p_1,p_2) 2^{-n} \|g\|_{p_2} \|Db\|_{p_1}\,.
\]
(ii) 
\[
\Big\|\sum_{k\in \cZ} S_k [g, \psi_{k+n}*  b]
 \Big \|_{1} \le C n 2^{-n} \|g\|_{\infty} \|Db\|_{H^1}\,.
\]
\end{prop} 
The proofs of the two propositions will be given in \S\ref{proofofmainfirst} and \S\ref{proofofeasy}.

\begin{remark}
Our proofs will show that if the index set $\cZ$ is a subset of $\bbZ_+$ then the Hardy space $H^1$ in Propositions \ref{mainfirst} and \ref{easy} can be replaced by the local Hardy space 
$h^1$ (cf. \ref{hardyremarks} below).
\end{remark}

\begin{remark}  
 The condition $\text{div}( b)=0$ is crucial for
Proposition \ref{mainfirst} but not needed for Proposition \ref{easy}. There are also $L^{p_1}\times L^{p_2}\to L^q$ estimates for other exponents with $p_1^{-1}+p_2^{-1}=q^{-1}$ in Proposition \ref{mainfirst} but they will not be relevant for Theorem \ref{bressan-bianch}.
\end{remark}

\begin{remarks}\label{hardyremarks} {\it On  Hardy spaces and atomic decompositions.} 
The proof of the Hardy space inequalities will  rely on the atomic decomposition (see e.g. \cite{stein3} for an exposition and historical references).
Let $1<r\le \infty$. 
We say that $a$ is an $r$-atom associated with a cube $Q$ if $a$ is supported in $Q$, if
$\|a\|_{L^r(Q)}\le |Q|^{-1+1/r}$ and if $\int a(x) dx=0$. Note that $\|a\|_1\le 1$ for atoms.
The atomic characterization of  $H^1$ states that any $f\in H^1$ can be decomposed as $f=\sum_Q \la_Q a_Q$ with convergence in $L^1$,
 where $a_Q$ are $r$-atoms and $\sum_Q|\la_Q|<\infty$.  
 The norm 
$\|f\|_{H^1}$ is equivalent to $\inf \sum_Q|\la_Q|$ where the infimum is taken over all such  decompositions of $f$. We shall assume $r<\infty$. 
An operator $T$ maps $H^1(\bbR^d)$ to $L^1(\bbR^d)$ if and only  we have $\|Ta\|_1\lc C$ for all  $r$-atoms; the 
infimum over such  $C$ is equivalent to the $H^1\to L^1$ operator norm of $T$. 
We refer to 
\cite{bownik}, \cite{MSV} for the reason why it is preferable to  work with $r$-atoms for $r<\infty$.

For compact manifolds  the appropriate Hardy space is the local Hardy space $h^1$,
introduced  by Goldberg \cite{goldberg}, which can be identified with the  Triebel-Lizorkin space $F^0_{p,q}$ for  $p=1$ and $q=2$, \cite{Tr83}. Functions in $h^1$ can be localized, i.e. if $f\in h^1$ and if $\chi\in C^\infty_0$ then $\chi f\in h^1$.
More generally, classical  pseudo-differential operators of order $0$ are bounded on $h^1$ (see \cite{goldberg}). 
Finally an operator $T$ maps $h^1 $ to $L^1$ if 
we have $\|Ta\|_1\lc C$ for all  $r$-atoms associated to  cubes with diameter $\le c_0$
and if in addition  $\|Tb\|_1\lc 1$ for all $L^r$ functions $b$ with
 $\|b\|_r \le 1$, which are supported on sets of bounded diameter.
\end{remarks}
\section{Proof of Proposition \ref{mainfirst}}\label{proofofmainfirst}
We shall use the $T1$ theorem of David and Journ\'e \cite{david-journe}.
For each term $S_k[g, \phi_k*b]$ we use the identity \eqref{KiKij} with  
$\phi_k*b$ in place of $b$, and with $\phi_k*\beta_{ij}$ in place of $\beta_{ij}$.
This reduces matters to the estimate of a singular integral operator $T\equiv T_{[g]}$ which acts 
on functions $h$,  and is,  
for fixed $g\in L^\infty$, defined by
\Be\label{Tdef}Th(x)=  \sum_{k\in \cZ} \int \chi_k(x-y) \ka(x-y) g(y)\int_0^1 \phi_k*h(sx+(1-s)y) ds \, dy.\Ee
Here $\ka$ is smooth away from the origin, homogeneous of degree $-d$, with mean value $0$ over $S^{d-1}$; in particular it can be any of the kernels in \eqref{Ki}, \eqref{Kij}.
Proposition \ref{mainfirst} follows from  the  inequalities
\begin{align}
\label{TLp}
\|Th\|_p &\lc \|g\|_\infty \|h\|_p,
\\
\label{Thardy}
\|Th\|_1 &\lc \|g\|_\infty \|h\|_{H^1}.
\end{align}

We now have to verify the hypothesis of   the David-Journ\'e theorem \cite{david-journe}.
Let $K$ be the Schwartz kernel of $T$, i.e. we have 
$$Th(x)=\int K(x,z) h(z) dz$$ for $h\in L^1+L^\infty$;  by our assumption on the index set $\cZ$ we see  $K(x,\cdot)$ is bounded and compactly supported (although $\cZ$ and these assumptions are not supposed to quantitatively enter 
in our estimates). We need to check that 
 $K$ and its derivatives  satisfy standard bounds for singular kernels, which are controlled by the $L^\infty$ norm of $g$;  {\it i.e.}
\Be\label{size}
|K(x,z)|\lc \|g\|_\infty |x-z|^{-d}
\Ee
and
\Be\label{deriv}
|\nabla_x K(x,z)| + |\nabla_z K(x,z)|
\lc \|g\|_\infty |x-z|^{-d-1}\,.
\Ee

Secondly, $T$ needs to satisfy  the {\it weak  boundedness property}. Let $\cN$ be the class of $C^1$ functions supported in $\{x:|x|\le 1\}$ such that 
$\|u\|_\infty+\|\nabla u\|_\infty \le 1$. 
For $u\in \cN$ define the translated and dilated versions $u_R^w$, $R>0$, $w\in \bbR^d$, by
 $u_R^w(x)= u(R^{-1}(x-w))$. Then we need to verify  for all $u,\tilde u\in \cN$
\Be\label{wbp}
\sup_{w\in \bbR^d} \sup_{R>0} R^{-d} \big|\biginn{ Tu^w_R}{\tilde u^w_R}\big |
\lc \|g\|_\infty\,.
\Ee

Finally,   we need  the crucial $BMO$-conditions
\Be\label{bmo}
\|T1\|_{BMO} + \|T^*1\|_{BMO} \lc \|g\|_\infty.
\Ee

\bigskip

We begin by  checking \eqref{size} and \eqref{deriv}.
We have $K(x,z)=\sum_k K_k(x,z)$ where 
$$
K_k(x,z)= \int \chi_k(x-y)\ka(x-y) g(y) \int_0^1 \phi_k(sx+(1-s)y-z) ds \, dy.
$$
Observe that $$K_k(x,z)=0 \text{  for $|x-z|\ge C 2^{-k}.$} $$
It is  immediate from the definition that 
$$|K_k(x,z)| \lc 2^{kd} \|g\|_\infty$$
and 
$$|\nabla_x K_k(x,z)|+|\nabla_z K_k(x,z)| \lc 2^{k(d+1)} \|g\|_\infty.$$
Fix $x,z$ and sum over $k$ with $2^k\lc |x-z|^{-1}$, and \eqref{size} and \eqref{deriv} follow.

Next, we check the weak boundedness property \eqref{wbp}. Let $T_k$ denote the operator with Schwartz kernel 
$K_k$. We estimate 
$\biginn{ T_ku^w_R}{\tilde u^w_R} $ 
and distinguish the cases $2^k R\le 1$ and $2^k R\ge 1$.

Write 
\begin{align*}
&
\biginn{ T_ku^w_R}{\tilde u^w_R}=
\iint K_k(x,z) u^w_R(z)\tilde u^w_R(x) \, dz\, dx\\&=
\iint\limits
\int \chi_k(x-y)\ka(x-y) g(y) \int_0^1\phi_k(sx+(1-s)y-z) ds \, dy
\\& \qquad\qquad\qquad\qquad\qquad\qquad\qquad\qquad\qquad\quad
\times u^w_R(z)  \,  \tilde u^w_R(x)  \, dx\,dz
\end{align*}
and since we have the conditions
$|x-w|\lc R$,
$|z-w|\lc R$,
$|y-x|\lc 2^{-k}$ for the domains of integration, a straightforward estimation yields 
\[\big |\biginn{ T_ku^w_R}{\tilde u^w_R} \big|
\lc 2^{kd} R^{2d} \|g\|_\infty \,\text{ if } R\le 2^{-k}.
\]

For $R\ge 2^{-k}$ we use that the integrals of $\ka$  over spheres centered at the origin  are zero. 
Since $\chi_k$ is radial we also have
\Be\label{canc} \int \chi_k(x) \ka(x)\, dx=0,\Ee
 for all $k\in \bbZ$.
We may write (after performing a change of variable)
\begin{align}
\notag
& \iint K_k(x,z) u^w_R(z)\tilde u^w_R(x) \, dx\, dz\\
\label{brackets}&=
\int g(y)\int_0^1 \int \phi_k((1-s)y-z)\big[\cdots\big] 
dz\, ds\, dy
\end{align}
where
\begin{align*}
&\big[\cdots\big]= \int u_R^w(z+sx) \tilde u_R^w(x) \chi_k(x-y)\ka(x-y) \, dx
\\
&=
\int \big(
u_R^w(z+sx) \tilde u_R^w(x) 
-u_R^w(z+sy) \tilde u_R^w(y) \big)
\chi_k(x-y)\ka(x-y) \, dx 
\\&= O(2^{-k}R^{-1}).
\end{align*}
Here we have of course used 
the cancellation property \eqref{canc}.
Using  the last  estimate in \eqref{brackets} we see that
\begin{align*}
\big|\biginn{ T_ku^w_R}{\tilde u^w_R}\big | &\lc (2^k R)^{-1}
\|g\|_\infty 
\int_{|y-w|\le CR} \int_0^1\int | \phi_k((1-s)y-z)|\, dz \, ds\, dy
\\ &
\lc (2^k R)^{-1} R^d
\|g\|_\infty \, \text{ if } R\ge 2^{-k}\,.
\end{align*}
Summing in $k$ over $2^{-k}\le R$ yields \eqref{wbp}.

Finally we need to verify the $BMO$ bounds for  $T1$ and $T^*1$.
First,
\begin{align*}
T_k 1(x) &= \int K_k(x,z) dz\\&= 
\int \chi_k(x-y) \kappa(x-y) g(y) \int_0^1 \int \phi_k(sx+(1-s)y-z) dz \, ds\, dy
\\&= (\chi_k \ka)* g(x).
\end{align*}
In view of the assumptions on $\ka$ the operator 
$g\mapsto \sum_k (\chi_k \ka) *g =\ka*g$ is a standard Calder\'on-Zygmund convolution operator and thus bounded from $L^\infty\to BMO$.
Thus  we get
$$\|T1\|_{BMO} \lc \|g\|_\infty.$$
Next,
\begin{align*} 
T_k^*1(z)&=\int K_k(x,z)dx
\\
&=\int\int_0^1\int  \chi_k(x-y)\ka(x-y) g(y)  \phi_k(sx+(1-s)y-z) dx\,ds\, dy
\\
&=\int g(y) \int \int s^{-d} \chi_k(s^{-1}(w-y))\ka(s^{-1}(w-y))\phi_k(w-z) \, dw\, ds\, dy
\end{align*}
where for fixed $y,s$ we changed variables $w=sx+(1-s)y$.

Hence setting $\ka_{k,s}(x)= \chi_k (s^{-1}x)s^{-d} \ka(s^{-1}x)$, we have
\[T^*1 = \sum_{k\in \cZ} \phi_k*\int_0^1 \ka_{k,s} ds\,  *g\,.\]

For fixed $s$ we use the cancellation of $\chi_k\ka$ to get an estimate for the Fourier transform
  of $\ka_{k,s}$,
  \[|\widehat {\ka_{k,s}} (\xi) | \le C_N s2^{-k}|\xi| (1+s2^{-k}|\xi|)^{-N}\,.\]
  It follows that $\sup_{\xi,s} \sum_k|\widehat {\ka_{k,s}}(\xi)|\le C$ and since $\widehat \phi_k =O(1)$ we see that 
     the Fourier transform of $\sum_k \phi_k *\ka_{k,s}$ is bounded, independently of $s$. Integrating  over  $s\in [0,1]$ we see that
\[
\Big\|\sum_{k\in \cZ} \phi_k*\int_0^1 \ka_{k,s} ds\,  *f\Big \|_2\lc \|f\|_2.\]
It is also clear that 
the convolution kernel satisfies  standard size and differentiability estimates 
in  Calder\'on-Zygmund theory and consequently  we get $L^\infty\to BMO$ boundedness. It follows that 
$$\|T^*1\|_{BMO} \lc \|g\|_\infty$$
and 
\eqref{bmo} is proved. This completes the proof of the $L^p$ estimates \eqref{TLp}.

The Hardy space estimate \eqref{Thardy} follows from the corresponding estimates on atoms
which are standard \cite{stein3}. For completeness we include the argument. Let $a$ be a $2$-atom associated with a cube $Q$ centered at $y_Q$ and let $Q^{*}$ be the triple cube. Then
$$\int_{Q^*}|Ta(x)|\, dx 
\le |Q^*|^{1/2} \|Ta\|_2 \lc
|Q^*|^{1/2} \|g\|_\infty \|a\|_2 \lc \|g\|_\infty.
$$
Since  $\int a(y) dy=0$ we get 
$$
\int_{\bbR^d\setminus Q^*}|Ta(x)|\, dx 
= \int_{\bbR^d\setminus Q^*}
 \int (K(x,y)-K(x,y_Q)) a(y) \,dy\, dx \lc
 \|g\|_\infty
 $$
given the size and derivative
assumptions in \eqref{size} and \eqref{deriv} and $\|a\|_1\le 1$. This finishes the proof of 
Proposition \ref{mainfirst}. \qed

\section{Proof of Proposition  \ref{easy}}\label{proofofeasy}
This will be  straightforward from standard estimates for singular convolution operators. 
Let
\[ \cK_{i,k}(x)= \chi_k(x) \frac{2^{-k}x_i}{|x|^{d+2}}.\]
We observe the commutator relation 
\Be \label{comm}
S_k[g, h](x) =2^k  \sum_{i=1}^d \big(
\cK_{i,k} * g(x) \,h_i(x) 
-\cK_{i,k} * [g h_i](x) \big)
\Ee
which we use with the choice $h_i=\psi_{k+n} *b_i$.
Notice that $\cK_{i,k}$ is an odd kernel and therefore
\Be\label{dualitySk}
\int \cK_{i,k}\!*\!g(x)\, h_i(x) f(x) \, dx=
- \int \cK_{i,k}\!* \![fh_i](x) \,g(x) \, dx
\Ee
Hence, in order to prove part (i) of  Proposition \ref{easy} it suffices to show
\begin{multline} \label{part-i}
\Big|\sum_{k\in \cZ} 2^k \int \cK_{i,k}\!*\!g(x) \,\psi_{k+n}\!*\!b_i(x) \, f(x) dx\Big|
\lc 2^{-n} \|f\|_{p_1}\|g\|_{p_2} \|\nabla b_i\|_{p_3}, \\ \text{ with } 
p_1^{-1}+p_2^{-1}+p_3^{-1}=1
 \text{ and }1<p_1,p_2,p_3<\infty\,.
\end{multline}
Moreover, to prove part (ii) it suffices to show
\Be\label{part-ii}
\Big|\sum_{k\in \cZ} 2^k \int \cK_{i,k}\!*\!g(x) \,\psi_{k+n}\!*\!b_i(x) \, f(x) dx\Big|
\lc n2^{-n} \|f\|_{\infty}\|g\|_{\infty} \|\nabla b_i\|_{H^1}.
\Ee

We first simplify  by rewriting the left hand sides as an expression which acts on 
 $\nabla b_i$. Let $\phi$  be as in \eqref{intphi}, \eqref{intphimoment} 
 and define
 for $j=1,\dots d$
 \begin{align*}
 \Psi^{[j]}(x)=& \int_{-\infty}^{x_j} 2^j \phi(2x_1,\dots, 2x_{j-1}, 2s, x_{j+1},\dots, x_d)\, ds
\\& -
 \int_{-\infty}^{x_j} 2^{j-1} \phi(2x_1,\dots, 2x_{j-1}, s, x_{j+1},\dots, x_d)\, ds
 \end{align*}
 Since $\phi$ is supported in $[-1/2,1/2]$ 
 it is then easy to check using \eqref{intphi} that $\Psi^{[j]}$ is also supported 
  in $[-1/2,1/2]$; moreover from  \eqref{intphimoment} and integration by parts we get
  $$\int \Psi^{[j]}(x) dx=0.$$ Now 
  let $\Psi_l^{[j]}(x)= 2^{ld}\Psi^{[j]}(2^l x)$, and we  
 verify that
    $$\psi_l= 2^{-l} \sum_{j=1}^d \frac{\partial \Psi_l^{[j]}}{\partial x_j}.$$
Thus by integration by parts 
$$\psi_{k+n}*b_i= 2^{-k-n} \sum_{j=1}^d\Psi^{[j]}_{k+n} *\frac{\partial b_i}{\partial x_j}\,.$$

Let $\Psi$ be any smooth function supported in $[-1/2,1/2]^d$ such that $\int\Psi(x) dx=0$,
and $\Psi_l=2^{ld}\Psi (2^l\cdot)$. The above  considerations imply   that in order to establish  
\eqref{part-i}, \eqref{part-ii}
 it suffices to prove
\Be\label{part-i-mod}
\Big|\sum_{k\in \cZ}  \int \cK_{i,k}\!*\!g(x) \,\Psi_{k+n}\!*\!h(x) \, f(x) dx\Big|
\lc  \|f\|_{p_1}\|g\|_{p_2} \|h\|_{p_3},\Ee 
with $\frac 1{p_1}+\frac 1{p_2}+\frac 1{p_3}=1,$ and $1<p_1,p_2,p_3<\infty,$ and 
\Be\label{part-ii-mod}
\Big|\sum_{k\in \cZ}  \int \cK_{i,k}\!*\!g(x) \,\Psi_{k+n}\!*\!h(x) \, f(x) dx\Big|
\lc n \|f\|_{\infty}\|g\|_{\infty} \|h\|_{H^1}.
\Ee

\subsubsection*{\it Proof of \eqref{part-i-mod}}
We apply H\"older's inequality several times and dominate the left hand side of \eqref{part-i-mod}
by
\begin{align}
&\|f\|_{p_1}\Big \| \sum_{k\in \cZ} (\cK_{i,k} *g )  (\Psi_{k+n}*h) \Big\|_{p_1'} 
\notag
\\
& \le \|f\|_{p_1} 
\Big \| \Big(\sum_k |\cK_{i,k} *g |^2\Big)^{1/2} \Big(\sum_k| \Psi_{k+n}*h|^2\Big)^{1/2} \Big\|_{p_1'} 
\notag
\\
& \le \|f\|_{p_1}
\Big \| \Big(\sum_k |\cK_{i,k} *g |^2\Big)^{1/2}\Big \|_{p_2} \Big\|\Big(\sum_k| \Psi_{k+n}*h|^2\Big)^{1/2} \Big\|_{p_3} 
\label{tripleofnorms}
\end{align}
where we have used $1/p_1'=1/p_2+1/p_3$.

For any bounded sequence $\gamma=\{\gamma_k\}$ with $\|\gamma\|_\infty\le 1$,
$\sum_k \gamma_k \cK_{i,k} $ defines a standard Calder\'on-Zygmund convolution kernel in $\bbR^d$  with bounds uniformly in $\gamma$. In particular we may randomly choose $\gamma=\pm 1$ and by 
the standard  averaging argument using Khinchine's inequality for Rademacher functions (see e.g. \cite[ch. II.5]{steinbook})
(or alternatively, arguments for vector-valued Calder\'on-Zygmund operators, cf. \cite[Appendix D]{steinbook}) 
we get the inequality
\Be \label{Sqfct}\Big\| \Big(\sum_{k} | \cK_{i,k} * g|^2\Big)^{1/2} \Big\|_{p_2} \le C(p_2) \|g\|_{p_2}, \Ee
for $1<p_2<\infty$.
Similarly, we also have  the Littlewood-Paley inequality (\cf. \cite[ch. II.5.] {steinbook})
\Be\label{Sob}
\Big\|\Big(\sum_{l\in \bbZ}\big| \Psi_l*h\big|^2\Big)^{1/2}\Big\|_{p_3}\le \widetilde C(p_3) \|h\|_{p_3}, 
\Ee
for $1<p_3<\infty$.
Now \eqref{part-i-mod} follows by using \eqref{Sqfct} and \eqref{Sob} in 
\eqref{tripleofnorms}. \qed

\subsubsection*{\it Proof of \eqref{part-ii-mod}}
Let $r\in (1,\infty)$.  It suffices to prove
\eqref{part-ii-mod} for $h=a$ with $a$ an $r$-atom associated to a cube $Q$. 
 Let $y_Q$ be the center of $Q$ and $Q^*$ be the 
double  cube with same center. Let $Q^{**}$ be the expanded cube with tenfold sidelength.
Let $L$ be such that the side length of $Q$ is between $2^{-L}$ and $2^{-L+1}$. 
We need to prove that
\Be\label{part-ii-atom}
\Big\|\sum_{k\in \cZ}  ( \cK_{i,k}*g)\,(\Psi_{k+n}*a)\,\Big\|_1
\lc n \|g\|_{\infty}.
\Ee

We split the sum in $k$ in three parts, according to 
whether 
$k\ge L$, $L-n\le k\le L$ or $k\le L-n$.


First let  $k>L$.  The support properties of $a$, $\Psi_{k+n}$ and $\cK_{k,i}$ show 
that $\Psi_{k+n}*a$ is supported in $Q^*$ and that
$\cK_{k,i}\!*\![g\bbone_{\bbR^d\setminus Q^{**}}](x)=0$ for $x\in Q^*$.
Hence 
$$\Psi_{k+n}\!*\!a (x)\,\cK_{k,i}\!*\!g(x)=\Psi_{k+n}\!*\!a (x)\,\cK_{k,i}\!*\![g\bbone_{Q^{**}}](x)
$$
in this case.
We choose $p_2, p_3\in (1,\infty)$ such that   $1/p_2+1/p_3+1/r=1$, and $p_3\le r$; for example
$p_2=p_3=r=3$. Now  use the already proven estimate \eqref{part-i-mod} together with H\"older's inequality to get \begin{align*}
&\Big\|
\sum_{\substack{k\in \cZ\\ k> L}}
 ( \cK_{i,k}\!*\!g) \,(\Psi_{k+n}\!*\!a)\Big\|_1
\lc |Q^*|^{1/r} \Big\|\sum_{\substack{k\in \cZ\\ k\ge L}}
 ( \cK_{i,k}\!*\![g\bbone_{Q^{**}}]) \,(\Psi_{k+n}\!*\!a)\Big\|_{r'}
\\
&\lc |Q^*|^{1/r} \|g\bbone_{Q^{**}}\|_{p_2} \,\|a\|_{p_3}
\lc |Q^*|^{1/r} \|g\|_\infty |Q^{**}|^{1/p_2} \,|Q|^{1/p_3-1/r}\|a\|_{r} \lc \|g\|_\infty
\end{align*}
since $\|a\|_r\le |Q|^{-1+1/r}$.

Next for the case $L-n\le k\le L$ we use the straightforward bound
$$\| (\cK_{i,k}\!*\!g) \,(\Psi_{k+n}\!*\!a)\|_1\le
\| \cK_{i,k}\!*\!g\|_\infty\|\Psi_{k+n}*a\|_1 \lc \|g\|_\infty \|a\|_1 \lc \|g\|_\infty\,
$$ and then obtain
$$\Big\| 
\sum_{\substack{k\in \cZ\\ L-n\le k\le L}}
(\cK_{i,k}\!*\!g) \,(\Psi_{k+n}\!*\!a)\Big\|_1\lc n \|g\|_\infty.
$$
Finally, if $k<L-n$ we use $\int a(x)dx=0$ to get
$$\Psi_{k+n}*a = \int \big(\Psi_{k+n}(x-y)-\Psi_{k+n}(x-y_Q) \big) a(y) dy$$
and thus $\|\Psi_{k+n}*a\|_1 \lc 2^{k+n-L} \|a\|_1$.
Hence 
\begin{align*}&\Big\| 
\sum_{\substack{k\in \cZ\\ k< L-n}}
(\cK_{i,k}\!*\!g) \,(\Psi_{k+n}\!*\!a)\Big\|_1
\le \sum_{\substack{ k< L-n}}\|\cK_{i,k}\!*\!g \|_\infty \|\Psi_{k+n}\!*\!a\|_1
\\
&\lc  \|g\|_\infty \sum_{k\le L-n} 2^{k+n-L} \|a\|_1 \lc \|g\|_\infty.
\end{align*}
We combine the three cases and obtain \eqref{part-ii-atom}. 
This completes the proof of Proposition  \ref{easy}. \qed

\section{Additional Remarks}\label{additional}

\subsection{\it On the  result by Crippa and de Lellis}
Corollary \ref {bressanhardy} can also be proved by a modification of the  approach by Crippa and deLellis. The elegant argument outlined in \cite[\S 8]{deLellis} reduces matters to an estimate
 for vector fields $x\mapsto b(x)$, namely 
\Be\label{diffquotient} 
\frac{|b(x)-b(y)|}{|x-y|}
 \le \fM b(x)+\fM b(y)
\Ee
where $\fM$ is a maximal operator to be determined, with 
\Be\label{hardyfM} \|\fM  b\|_{L^1} \lc \|\nabla b\|_{h_1}\,.\Ee

Assume that $|x-y|\le 10^{-2}$.
Now let  $\phi\in C^\infty_c$ supported on $\{y:|y|\le 1/4\}$  such that $\int\phi(y) \, dy=1$,
and $\int y_i\phi(y) dy=0$ for $i=1,\dots,d$. Let 
$\phi_k(x)=2^{kd}\phi(2^kx)$, and
$\psi_k=\phi_{k}-\phi_{k-1}$ so that 
for any $\ell>0$, $$b=\phi_\ell*b+\sum_{k=\ell+1}^\infty\psi_k* b.$$
Now assume $2^{-\ell-1}\le |x-y|\le 2^{-\ell}$. 
\begin{align*}\frac{|\phi_\ell*b(x)-\phi_\ell* b(y)|}{|x-y|}&=\Big| \biginn{\frac{x-y}{|x-y|}}{\int_0^1   \phi_\ell *\nabla b((1-s)x+sy) }\,ds\Big|
\\&\le \cM_0 (\nabla b)(x)+ \cM_0(\nabla b)(y)
\end{align*}
where
$$\cM_0g(x) = \sup_{\ell>4} \sup_{|h|\le 2^{-\ell}}  |\phi_\ell *g(x+h)|.$$
By standard Hardy space theory,
$$\|\cM_0  g\|_{L^1} \lc \|g\|_{h_1}$$
(which will be  applied here to $g=\partial b_i/\partial x_j$).

Secondly, for $k\ge \ell$, 
\begin{align*}\frac{|\psi_k*b(x)-\psi_k *b(y)|}{|x-y|}&\le 2^{\ell+2} \sup_k\big(|\psi_k*b(x)| +|\psi_k*b(y) |\big)\\
&\le M_{1} b(x)+M_1b(y)\end{align*}
with $$M_1b(x)=\sup_{k>0} 2^k|\psi_k*b(x)|\,.$$
Now, by the cancellation property of $\psi$, $\int \psi(y)l(y) dy=0$ for all affine linear functions $l$, we have
$$\|M_1 b\|_1 \le \Big\|\Big(\sum_{k=1}^\infty 2^{2k} |\psi_k*b|^2\Big)^{1/2}\Big\|_1 \lc
 \|\nabla b\|_{h^1};$$
 in fact  by definition of   $M_1$ we have  the better estimate in terms of the Triebel-Lizorkin $F^0_{1,\infty}$-norm of $\nabla b$ (\cf. \cite{Tr83}).  We have now proved 
 \eqref{diffquotient} with $\fM b=\cM_0 (\nabla b)+M_1(b)$   and $\fM$ satisfies \eqref{hardyfM}.

\subsection{\it On  L\'eger's result for transport equations}\label{legersection}  In a recent preprint 
L\'eger  \cite{leger}  considers
solutions $\theta(t,x)$ of the initial value problem \begin{gather*} \partial_t\theta+\text{div}(v\theta)=0\\ \theta(0,\cdot)=\theta_0\end{gather*} on $\bbR^d$; here $v$ is a given divergence-free time-dependent  vector field $v$ on $ [0,\infty)\times \bbR^d$.  See  also  \cite{ltd}, \cite{ikx} for related  versions  of the mixing problem. L\'eger introduces  the functional
$$\cV(f)= \int|\widehat f(\xi)|^2 \log|\xi| d\xi$$ 
which in physical space  is computed to
\begin{multline*}
c_1(d) \Big(\frac 12\iint_{|x-y|\le 1} \frac{|f(x)-f(y)|^2}{|x-y|^d} dx\, dy-
\iint_{|x-y|\ge 1}\frac{f(x) f(y)}{|x-y|^d}\Big)+c_2(d)\|f\|_{L^2}^2
\end{multline*} for suitable constants $c_i(d)$.
He  then shows that 
 \Be \label{legerapp} 
\partial_t \cV(\theta(t,\cdot))=c_d\,
\fS[ \theta(t,\cdot), \theta(t,\cdot) , v(t,\cdot)]
\Ee 
with $\fS$ as in \eqref{fSepsRdef}, \eqref{pv}.
This is closely related to the computation in  Proposition \ref{bianchprop}. 
Note that L\'eger's  reduction to an estimate for $\fS$  works for arbitrary initial data $\theta_0$ while  Proposition \ref{bianchprop}
is limited to indicator functions of sets. L\'eger  uses the results in \cite{sss} (\cf. \S \ref{cj-reduction} above)   to dominate,
for $\theta(t,\cdot)\in L^\infty\cap L^{p'}$,  the right hand side of \eqref{legerapp} by  
$\|\theta(t,\cdot)\|_\infty\|\theta(t\cdot)\|_{p'} \|Dv(t, \cdot)\|_p$.
Our estimate \eqref{trilinearformest} yields the endpoint bound
\Be\label{legerappH1}|\partial_t \cV(\theta(t,\cdot))|\le C_d\|\theta(t, \cdot)\|_\infty^2 \|Dv(t, \cdot)\|_{H^1}\,.\Ee
This inequality can be used to extend other results in  \cite{leger}. For example one obtains  the inequality
$$\cV(\theta(t,\cdot))-\cV(\theta_0)\lc \|\theta_0\|_\infty^2 \int_0^t \|Dv(s,\cdot)\|_{H^1} ds\,.
$$

\section{Failure of a singular integral estimate}
\label{nonL1}
Deviating slightly from our previous notation in \eqref{halftorus} we now let $\Omega_L=(-1, 0)\times(-1,1)$, $\Omega_R=(0,1)\times(-1,1)$.
For a resolution of Bressan's problem on $\bbT^2$  it would be relevant if the inequality 
\Be\label{mixtranslation}
\Big |\iint \frac{\inn{x-y}{b(x)-b(y)}}{|x-y|^{4}} 
\chi_A(x)\chi_B(y) \,dx\, dy\Big| \le C(A,B) \|Db\|_1
\Ee
held for   subsets $A\subset \Omega_L$, $B\subset \Omega_R$ and divergence free vector fields 
$b$, with a constant independent of $A$ and $B$.
In particular we could consider regularized versions of 
$$
b(x)=\begin{cases} (0,1) \text{ for } x_1<0\,,
\\
(0,-1)\text{ for } x_1>0\,.
\end{cases}
$$
Notice that $$Db(x) = \begin{pmatrix} 0&0\\-2 \delta(x_1)&0
\end{pmatrix}$$
where $\delta$ is the Dirac measure in one dimension, and 
thus div(b)=0.
For this choice of $b$ the  expression \eqref{mixtranslation} becomes $|\cI(A,B)|$ with 
\begin{subequations} 
\Be
\cI(A,B)= 
\iint_{(x,y)\in A\times B} K_{|x_1-y_1|}(x_2-y_2)  dx\, dy
\Ee
where 
\Be K_r(s)= \frac{s}{(r^2+s^2)^2}= -\frac{1}{2} \frac{d}{ds}\frac{1}{r^2+s^2}.\Ee
\end{subequations}
We show that 
$\cI(A,B)$ is {\it not} bounded independently of $A\subset \Omega_L$, $B\subset \Omega_R$. 
One gets  a precise upper and lower bound in terms of some separation condition 
on $A$ and $B$.
\begin{prop}
Let  $$\cU(\eps) = \sup\,\big \{ |\cI(A,B)|\,:\,\,\dist (A,B)\ge \eps,\, A\subset \Omega_L,\,B\subset \Omega_R\big\}.$$
Then for $0<\epsilon<1/2$ we have
$$\cU(\eps) \approx \log (1/\eps).
$$
\end{prop}


\subsection{\it Upper bounds}
Suppose $g$ satisfies \Be \label{gassu}\sup_s(1+|s|)^{\delta+1}|g(s)|<\infty, \Ee for some $\delta>0$. 
Note that $K_r(s)= r^{-2} r^{-1}g(s/r)$ if we take $g(s)=s(1+s^2)^{-2}$,
 and thus the following estimate gives the upper bound in the proposition.
\begin{lemma} Suppose $A\subset \Omega_L$, $B\subset \Omega_R$ and  $\dist(A,B)>\eps$. Then, with 
$g$ as in \eqref{gassu},
$$\iint_{\Omega_L\times \Omega_R} |x_1-y_1|^{-3} |g(\tfrac{x_2-y_2}{|x_1-y_1|})| \chi_B(y)\chi_A(x) dx\, dy 
\lc \log(1/\eps).$$
\end{lemma}

\begin{proof}
Observe that for $x\in \Omega_L$, $y\in \Omega_R$ we have
$|x_1-y_1|=|x_1|+|y_1|$.

We  consider separately the regions with 
(i) $|x_2-y_2|\le |x_1-y_1|$ 
(which for $x\in A$, $y\in B$ implies $|x_1-y_1|\ge \eps/2$) and
 (ii) $2^{m-1}|x_1-y_1|\le  |x_2-y_2|<2^{m}|x_1-y_1|$ for some $m\ge 1$  (which for $x\in A$, $y\in B$ implies $|x_1-y_1|\ge 2^{-m-2} \eps$).

First,
\begin{align*}&\iint\limits_{\substack{(x,y)\in\Omega_L\times \Omega_R\\|x_2-y_2|\le |x_1-y_1|}}
|x_1-y_1|^{-3} |g(\tfrac{x_2-y_2}{|x_1-y_1|})| 
\chi_B(y)\chi_A(x) dx\, dy 
\\
\lc 
&\iint\limits_{\substack{(x_1,y_1)\in [-1,0]\times[0,1]\\ |x_1-y_1|\ge\eps/2}}
|x_1-y_1|^{-2} \iint_{[-1,1]^2}
\tfrac{1}{|x_1-y_1|} |g(\tfrac{x_2-y_2}{|x_1-y_1|})| 
dx_2dy_2\,\, dx_1dy_1
\\
\lc 
&\|g\|_{L^1(\bbR)} \iint_{\eps/2 <|x_1|+|y_1|\le 2}
\frac{1}{(|x_1|+|y_1|)^2} 
 dx_1dy_1\lc \log(1/\eps)
\end{align*}
Next, when $|x_2-y_2|\approx 2^m|x_1-y_1|$ we have 
$|g(\tfrac{x_2-y_2}{|x_1-y_1|})| \lc 2^{-m(1+\delta)}$ and thus 
\begin{align*}&\iint\limits_{\substack{(x,y)\in\Omega_L\times \Omega_R\\|x_2-y_2|\approx 2^m |x_1-y_1|}}
|x_1-y_1|^{-3} |g(\tfrac{x_2-y_2}{|x_1-y_1|})| 
\chi_B(y)\chi_A(x) dx\, dy 
\\
&\lc 2^{-m\delta}
\iint\limits_{\substack{(x_1,y_1)\in [-1,0]\times[0,1]\\ |x_1-y_1|\ge2^{-m-2}\eps}}
|x_1-y_1|^{-2} dx_1dy_1
\\
&\lc 2^{-m\delta}
\iint_{2^{-m-2}\eps \le |x_1|+y_1|\le2}
\frac{1}{(|x_1|+|y_1|)^2} 
 dx_1dy_1\lc  
2^{-m\delta}\log(2^m/\eps).
\end{align*}
Now sum in $m$ to finish the proof.
\end{proof}

\subsection{\it Lower  bounds}
We now take $g(s) =\frac{s}{(1+s^2)^2}$ and  construct a specific pair $A$, $B$ for which $dist(A,B)\ge \eps$ and $|\cI(A,B)|\gc \log(1/\eps)$.
It suffices to take $\eps=2^{-LM}$ for some integer $L$
(and $M$ be a sufficiently large fixed integer, $M>10$).

Define
\begin{align*}
I^L_k&=[-2^{-kM}, -2^{-kM-1}]\,,\\I^R_k&=[2^{-kM-1}, 2^{-kM}]\,,
\\
J^L_{k,n}&=[(Mn+2)2^{-kM}, (Mn+3)2^{-kM}]\,,
\\J^R_{k,n}&=[Mn2^{-kM}, (Mn+1)2^{-kM}]
\end{align*}
and 
\begin{align*}
A&= \bigcup_{1\le k\le L-1}\,\,\bigcup_{0\le n\le \frac{2^{kM}}{M+1}} 
I^L_k\times J^L_{k,n}\,,
\\
B&= \bigcup_{1\le k\le L-1}\,\,\bigcup_{0\le n\le \frac{2^{kM}}{M+1}}  I^R_k\times J^R_{k,n}\,.
\end{align*}
Observe that $B$ is a vertical translation of the horizontal reflection of $A$ and that both sets consist of columns of squares at $L-1$ many different scales.  Clearly  $A\in \Omega_L$, $B\in \Omega_R$ and $\dist(A,B)\ge 2^{-LM}$.

Let 
$$\cI(k_L, k_R, n_L, n_R)=\int_{x\in I^L_{k_L}\times J^L_{k_L,n_L}}
\int_{y\in I^R_{k_R}\times J^L_{k_R,n_R}} K_{|x_1-y_1|}(x_2-y_2)  \,dy\, dx$$
and split $\cI(A,B)=\cE_1+\cE_2+\cE_3$ where
\begin{align*}
\cE_1&=\sum_{1\le k\le L-1} \,\, \sum_{0\le n\le \frac{2^{kM}}{M+1}} \cI(k,k,n,n)
\\
\cE_2&=\sum_{1\le k\le L-1} \,\,
\sum_{\substack {0\le n_L, n_R\le\frac{ 2^{kM}}{M+1}\\n_L\neq n_R}} \cI(k,k,n_L,n_R)
\\
\cE_3&=\sum_{\substack{1\le k_L, k_R\le L-1\\k_L\neq k_R}} \,\,
\sum_{0\le n_L, n_R\le \frac{2^{kM}}{M+1}} \cI(k_L,k_R,n_L,n_R)\,.
\end{align*}
We prove a lower bound for $\cE_1$ and upper bounds for $\cE_2$, $\cE_3$.

For the lower bound 
observe 
$$2^{-kM}\le x_2-y_2\le 2^{2-kM} \text{ for $x_2\in J^L_{k,n}$, $y_2\in J^L_{k,n}$}.$$
Thus 
\begin{align*}
&\int_{I^L_{k}\times J^L_{k,n}}
\int_{I^R_{k}\times J^L_{k,n}} K_{|x_1-y_1|}(x_2-y_2)  dx\, dy
\\&=
\iiiint\limits_{\substack{
(x_1,y_1)\in [2^{-kM-1}, 2^{-kM}]
\\
(Mn+2) 2^{-kM}\le x_2\le(Mn+3)2^{-kM}
\\
Mn 2^{-kM}\le y_2\le(Mn+1)2^{-kM}}}
\frac{x_2-y_2}{((x_1+y_1)^2+(x_2-y_2)^2)^2} dy_2\,dx_2\,dy_1\, dx_1
\\
&\ge \frac{2^{-kM}}{1000}
\end{align*}
and thus 
$$\cE_1\ge 10^{-3}\sum_{k=1}^{L-1}
\sum_{0\le n_L\le\frac{ 2^{kM}}{M+1}} 2^{-kM}\ge
 \frac{L-1}{10^3(M+1)}.$$

If $n_L\neq n_R$ we have
$$
\int_{I^L_{k}\times J^L_{k,n_L}}
\int_{I^R_{k}\times J^L_{k,n_R}} |K_{|x_1-y_1|}(x_2-y_2)|  dx\, dy \lc \frac{2^{-kM}}{M^3|n_L-n_R|^3}
$$
and thus
$$|\cE_2|\le \sum_{1\le k<L} \sum_{0\le n_L\le\frac{ 2^{kM}}{M+1}}
\sum_{n_R\neq n_L} \frac{2^{-kM}}{M^3|n_L-n_R|^3}\,\le 
  C \frac{L}{M^4}.$$

Next, set $g(s)=|s|(1+s^2)^{-2}$, and
 $$ G(x_1,y_1)=\iint_{-1\le x_2,y_2 \le 1}
 \tfrac{1}{|x_1-y_1|} |g(\tfrac{|x_2-y_2|}{|x_1-y_1|})| dy_2dx_2 
$$
so that $G(x_1,y_1)$ is nonnegative  and uniformly bounded.
We have 
  $|\cE_3|\le  \cE_{3,1}+\cE_{3,2}$
where  
$$\cE_{3,1}=\sum_{1\le k_L< k_R\le L-1} 
\iint_{I^L_{k_L}\times I^R_{k,R}} |x_1-y_1|^{-2} G(x_1,y_1)
\,dy_1dx_1
$$
and 
$\cE_{3,2}$ is the corresponding term with the $(k_L, k_R)$ summation extended over $1\le k_R< k_L\le L-1$.
The two terms are symmetric and it suffices to estimate $\cE_{3,1}$.

Now  $|x_1-y_1|\approx 2^{-k_LM }$ if $x_1\in I^L_{k_L}$, 
$y_1\in I^L_{k_R}$,  and $k_L<k_R$.
Therefore
$$\cE_{3,1}\lc \sum_{1\le k_L< k_R\le L-1} 
2^{2k_L M} |I^L_{k_L}|\,|I^R_{k,R}|\lc 
\sum_{1\le k_L< k_R\le L-1} 
2^{(k_L-k_R) M} \lc 
L 2^{-M}
$$
and similarly we also get 
$\cE_{3,2}\lc L2^{-M}$. Combining the estimates we get
$$\cU(2^{-LM})\ge \cE_1-|\cE_2|-|\cE_3| \ge 10^{-3} \frac{L-1}{M+1} -
C_1 L M^{-3}- C_2 L 2^{-M}$$
 and the assertion follows by choosing $M$ sufficiently large.
\qed

\subsection{{\it A discrete problem}}

The counterexample suggests that to make progress towards the resolution of the $L^1$-conjecture, we need to first understand the effects of shear flows such as the vector field $b$ above.  To highlight this particular difficulty, we propose a simple discrete problem reminiscent of the Rubik's cube.

We mix the discrete torus $\Omega_n = \bbZ^2 / 2 n \bbZ^2$ by applying a sequence of sliding moves.  The goal is to transform the initial set
\begin{equation*}
  A_0 = [1,n] \times [1,2n] + 2 n \bbZ^2
\end{equation*}
into the final set
\begin{equation*}
  A_1 = \{ (x,y) \in \bbZ^2 : (-1)^{x+y} = 1 \}.
\end{equation*}
For integers $0 < b - a < 2 n$, consider the periodic strips $S \subseteq \bbZ^2$ given by
\begin{equation*}
  S = \bbZ \times ([a,b] + 2 n \bbZ)
\end{equation*}
and the permutation $P : \bbZ^2 \to \bbZ^2$ given by
\begin{equation*}
  P(x,y) = (x,y) + (1,0) 1_S(x,y).
\end{equation*}
Such permutations, when composed with an arbitrary number of $90^\circ$ rotations, are the allowed sliding moves.

For this simplified problem, a positive answer to the Bressan's mixing conjecture would imply that it takes at least $c n \log n$ sliding moves to transform $A_0$ into $A_1$.  It is clear from looking at the Cayley graph of the group generated by the finite set of sliding moves, that the diameter of the set of reachable configurations is much larger than $n \log n$.  However, Bressan's conjecture in this context is a statement about the minimal distance between two particular configurations $A_0$ and $A_1$.

\section{A toy problem on $\bbT^2$}\label{toymodel}
Consider the problem of mixing $\bbT^2$ by a finite sequence of $90^\circ$ rotations of squares. Given $x\in \bbT^2$ and $r\in  (0,1/4)$, let $R_{x,r} : \bbT^2 \to \bbT^2$ be the map which rotates the square $(x_1 - r, x_1 + r)\times (x_2 - r, x_2 + r)$  by $90^\circ$  counter-clockwise:
\[
R_{x,r}(y):=\begin{cases}
 (x_1+x_2-y_2,x_2-x_1+y_1) &\text{ if } y-x\in(-r,r)^2, 
 \\ y &\text{otherwise}.
 \end{cases}\]
We assign the {\it cost} $r^2$ to the rotation $R_{x,r}$. To motivate this definition observe that we can write $R_{0,r}(x)=X_r(1,x)$ where $X_r: [0,1]\times \bbT^2\to \bbT^2$ is the incompressible flow that satisfies
 \[ D_tX_r(t,x)= \begin{cases} (0,2x_1) &\text{ if $|x_2|<|x_1|<r$,}
 \\(-2x_2,0) &\text{ if $|x_1|<|x_2|<r$,}
 \\(0,0) &\text{ otherwise,} \end{cases}
 \]
 in the coordinates $(-1/2,1/2)^2$ for the torus $\bbT^2$. The vector field $D_t X_r(t,\cdot)$ is the weakly divergence free square vortex:
 
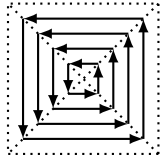
\begin{center}
\bigskip
\begin{tikzpicture}
\draw[thick,dotted] (1,-1) -- (1,1) -- (-1,1) -- (-1,-1) -- (1,-1);
\draw[thick,dotted] (1,1) -- (-1,-1);
\draw[thick,dotted] (1,-1) -- (-1,1);
\draw[thick,->,>=latex] (0.2,-0.2) -- (0.2,0.2);
\draw[thick,->,>=latex] (0.2,0.2) -- (-0.2,0.2);
\draw[thick,->,>=latex] (-0.2,0.2) -- (-0.2,-0.2);
\draw[thick,->,>=latex] (-0.2,-0.2) -- (0.2,-0.2);
\draw[thick,->,>=latex] (0.4,-0.4) -- (0.4,0.4);
\draw[thick,->,>=latex] (0.4,0.4) -- (-0.4,0.4);
\draw[thick,->,>=latex] (-0.4,0.4) -- (-0.4,-0.4);
\draw[thick,->,>=latex] (-0.4,-0.4) -- (0.4,-0.4);
\draw[thick,->,>=latex] (0.6,-0.6) -- (0.6,0.6);
\draw[thick,->,>=latex] (0.6,0.6) -- (-0.6,0.6);
\draw[thick,->,>=latex] (-0.6,0.6) -- (-0.6,-0.6);
\draw[thick,->,>=latex] (-0.6,-0.6) -- (0.6,-0.6);
\draw[thick,->,>=latex] (0.8,-0.8) -- (0.8,0.8);
\draw[thick,->,>=latex] (0.8,0.8) -- (-0.8,0.8);
\draw[thick,->,>=latex] (-0.8,0.8) -- (-0.8,-0.8);
\draw[thick,->,>=latex] (-0.8,-0.8) -- (0.8,-0.8);
\draw[thick] (1,1) coordinate (1);
\end{tikzpicture}
\bigskip
\end{center}
Let $M(\bbT^d)$ be the space of Borel measures on $\bbT^2$. 
Since $$\int_0^1 \|D_xD_t X_r(t,x)\|_{M(\bbT^2)} dt = C r^2$$
our choice for the cost is natural.
  The following result can therefore be considered to solve a discrete toy version of 
Bressan's conjecture. 
\begin{theorem}\label{discretethm} If $R_{x_1,r_1}\circ\cdots\circ R_{x_n,r_n} (0, 1/2)^2$ is mixed to scale $\eps\in (0,1/2)$, then
\Be \sum_{i=1}^n r_i^2 \ge C^{-1}\log \eps^{-1},\Ee 
with a universal constant $C >0$.
\end{theorem}

To see the sharpness of the result  consider the composition
$$R^3_{(\frac 14,\frac 12),\frac 14}\circ R_{(\frac 12,\frac 14), \frac 14}\circ R^2_{(\frac 12,\frac 12),\frac 14}$$
which divides $(0, 1/2)^2$ into four smaller squares, at cost $6r^2$:
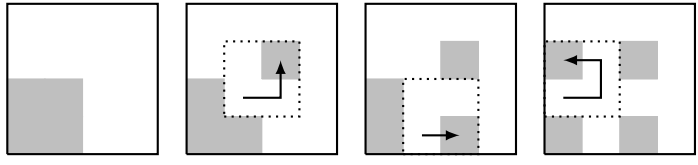
\begin{center}
\bigskip
\begin{tabular}{cccc}
\begin{tikzpicture}[scale=0.5]
\filldraw[lightgray] (0,0) -- (1,0) -- (1,1) -- (0,1) -- (0,0);
\filldraw[lightgray] (1,0) -- (2,0) -- (2,1) -- (1,1) -- (1,0);
\filldraw[lightgray] (0,1) -- (1,1) -- (1,2) -- (0,2) -- (0,1);
\filldraw[lightgray] (1,1) -- (2,1) -- (2,2) -- (1,2) -- (1,1);
\draw[thick] (0,0) -- (4,0) -- (4,4) -- (0,4) -- (0,0);
\end{tikzpicture}
&
\begin{tikzpicture}[scale=0.5]
\filldraw[lightgray] (0,0) -- (1,0) -- (1,1) -- (0,1) -- (0,0);
\filldraw[lightgray] (1,0) -- (2,0) -- (2,1) -- (1,1) -- (1,0);
\filldraw[lightgray] (0,1) -- (1,1) -- (1,2) -- (0,2) -- (0,1);
\filldraw[lightgray] (2,2) -- (3,2) -- (3,3) -- (2,3) -- (2,2);
\draw[thick] (0,0) -- (4,0) -- (4,4) -- (0,4) -- (0,0);
\draw[thick,dotted] (1,1) -- (3,1) -- (3,3) -- (1,3) -- (1,1);
\draw[thick,->,>=latex] (1.5,1.5) -- (2.5,1.5) -- (2.5,2.5);
\end{tikzpicture}
&
\begin{tikzpicture}[scale=0.5]
\filldraw[lightgray] (0,0) -- (1,0) -- (1,1) -- (0,1) -- (0,0);
\filldraw[lightgray] (2,0) -- (3,0) -- (3,1) -- (2,1) -- (2,0);
\filldraw[lightgray] (0,1) -- (1,1) -- (1,2) -- (0,2) -- (0,1);
\filldraw[lightgray] (2,2) -- (3,2) -- (3,3) -- (2,3) -- (2,2);
\draw[thick] (0,0) -- (4,0) -- (4,4) -- (0,4) -- (0,0);
\draw[thick,dotted] (1,0) -- (3,0) -- (3,2) -- (1,2) -- (1,0);
\draw[thick,->,>=latex] (1.5,0.5) -- (2.5, 0.5);
\end{tikzpicture}
&
\begin{tikzpicture}[scale=0.5]
\filldraw[lightgray] (0,0) -- (1,0) -- (1,1) -- (0,1) -- (0,0);
\filldraw[lightgray] (2,0) -- (3,0) -- (3,1) -- (2,1) -- (2,0);
\filldraw[lightgray] (0,2) -- (1,2) -- (1,3) -- (0,3) -- (0,2);
\filldraw[lightgray] (2,2) -- (3,2) -- (3,3) -- (2,3) -- (2,2);
\draw[thick] (0,0) -- (4,0) -- (4,4) -- (0,4) -- (0,0);
\draw[thick,dotted] (0,1) -- (2,1) -- (2,3) -- (0,3) -- (0,1);
\draw[thick,->,>=latex] (0.5,1.5) -- (1.5,1.5) -- (1.5,2.5) -- (0.5, 2.5);
\end{tikzpicture}
\end{tabular}
\bigskip
\end{center}
Applying this idea recursively, we see that we can mix to scale $2^{-n}$ at cost $C n r^2$.

\subsection*{\it Proof of Theorem \ref{discretethm}}

We use the Bianchini semi-norm defined in \S\ref{bianchini-section}.
\begin{lemma}
If $u : \bbT^d \to \bbT^d$ is measure preserving, $A \subseteq \bbT^d$, and $\|\bbone_A\|_\cB$ is finite, then
\begin{equation}
\label{Bdifference}
\|\bbone_{u(A)}\|_\cB - \|\bbone_A\|_\cB
 \leq \int_0^{1/4} \frac{1}{r |B_r(0)|} \int_{\bbT^d} \left| u(B_r(x)) \triangle B_r(u(x)) \right| \,dx \,dr.
\end{equation}
\end{lemma}

\begin{proof}
We compute $\|\bbone_{u(A)}\|_\cB - \|\bbone_A\|_\cB $ as 
\begin{align*}
&  \int_0^{1/4} \frac{1}{r} \int_{\bbT^d} \left| \chi_{u(A)}(x) - \fint_{B_r(x)} \chi_{u(A)}(y) \,dy \right| \,dx \,dr \\
& \qquad -  \int_0^{1/4} \frac{1}{r} \int_{\bbT^d} \left| \chi_A(x) - \fint_{B_r(x)} \chi_A(y) \,dy \right| \,dx \,dr \\
& =  \int_0^{1/4} \frac{1}{r} \int_{\bbT^d} \left| \chi_{A}(x) - \fint_{u^{-1}(B_r(u(x)))} \chi_A(y) \,dy \right| \,dx \,dr \\
& \qquad -  \int_0^{1/4} \frac{1}{r} \int_{\bbT^d} \left| \chi_A(x) - \fint_{B_r(x)} \chi_A(y) \,dy \right| \,dx \,dr \\
& \leq  \int_0^{1/4} \frac{1}{r} \int_{\bbT^d} \left| \fint_{u^{-1}(B_r(u(x)))} \chi_A(y) \,dy -  \fint_{B_r(x)} \chi_A(y) \,dy \right| \,dx \,dr \\
& \leq \int_0^{1/4} \frac{1}{r |B_r(0)|} \int_{\bbT^d} \left| u(B_r(x)) \triangle B_r(u(x)) \right| \,dx \,dr,
\end{align*}
using the fact that $u$ is measure preserving to change variables.
\end{proof}

\begin{lemma}
There is a constant $C > 0$ such that
\begin{equation}
\label{Brotations}
 \int_0^{1/4} \frac{1}{r |B_r(0)|} \int_{\bbT^2} \left| R_{x,s}(B_r(y)) \triangle B_r(R_{x,s}(y)) \right| \,dy \,dr \leq C s^2,
\end{equation}
for all $x \in \bbT^2$ and $s \in (0,1/4)$.
\end{lemma}

\begin{proof}
By scaling, observe that
\begin{multline*}
\int_0^s \frac{1}{r |B_r(0)|} \int_{\bbT^2} \left| R_{x,s}(B_r(y)) \triangle B_r(R_{x,s}(y)) \right| \,dy \,dr \\
\leq s^2 \int_0^{1/4} \frac{1}{r |B_r(0)|} \int_{\bbT^2} \left| R_{x,1/4}(B_r(y)) \triangle B_r(R_{x,1/4}(y)) \right| \,dy \,dr
= C s^2.
\end{multline*}
Next, observe that if $r \geq s$ and
\begin{equation*}
 \left| R_{x,s}(B_r(y)) \triangle B_r(R_{x,s}(y)) \right| > 0,
\end{equation*}
then either
\begin{equation*}
\left| R_{x,s}(B_r(y)) \triangle B_r(R_{x,s}(y)) \right| \leq C s^2 \quad \mbox{and} \quad r - \sqrt{2} s \leq |y - x| \leq r + \sqrt{2} s,
\end{equation*}
or
\begin{equation*}
\left| R_{x,s}(B_r(y)) \triangle B_r(R_{x,s}(y)) \right| \leq C s r \quad \mbox{and} \quad |y - x| < \sqrt 2 s.
\end{equation*}
In particular, we may estimate
\begin{multline*}
\int_s^{1/4} \frac{1}{r |B_r(0)|} \int_{\bbT^2} \left| R_{x,s}(B_r(y)) \triangle B_r(R_{x,s}(y)) \right| \,dy \,dr \\
\leq \int_s^{1/4} \frac{1}{r |B_r(0)|} C s^3 r \,dr \leq C s^2
\end{multline*}
Putting these two estimates together gives \eqref{Brotations}.
\end{proof}

\begin{proof}[Proof of Theorem \ref{discretethm}, conclusion]
If $A$ is mixed to scale $\eps \in (0,\kappa)$,  with mixing constant $\kappa$ then the average of $\bbone_A$ over $B_r(x)$ lies between
$\kappa |A| $ and $(1-\kappa)|A|$ when $r \geq \eps. $ Thus
\[
\|\bbone_A\|_{\cB}\geq \kappa\int_\eps^{\kappa} \frac{1}{r} \min \{ |A|, (1 - |A|) \} \ge  \frac{1}{C} \min \{ |A|, 1 - |A| \}  \log \eps^{-1}.
\]
Combine this with 
 \eqref{Bdifference}, and \eqref{Brotations} to conclude the proof.
\end{proof}

\noi{\it Remark.} This  $L^1$-type Bressan result for the toy problem  is possible since the natural scale 
$s$ for the rotation $R_{y,s}$ is linked in the proof with the scale $r$ in the Bianchini semi-norm,
with maximal contributions for $r\approx s$. 

\bibliographystyle{amsalpha}

\newpage
\end{document}